\documentclass[reqno, 12pt]{amsart}
\pdfoutput=1
\makeatletter
\let\origsection=\section \def\section{\@ifstar{\origsection*}{\mysection}} 
\def\mysection{\@startsection{section}{1}\z@{.7\linespacing\@plus\linespacing}{.5\linespacing}{\normalfont\scshape\centering\S}}
\makeatother        

\usepackage{amsmath,amssymb,amsthm}
\usepackage{mathrsfs}
\usepackage{mathabx}\changenotsign
\usepackage{dsfont}
 
\usepackage{xcolor}
\usepackage[backref]{hyperref}
\hypersetup{
    colorlinks,
    linkcolor={red!60!black},
    citecolor={green!60!black},
    urlcolor={blue!60!black}
}

\usepackage{graphicx}

\usepackage{verbatim}

\usepackage[open,openlevel=2,atend]{bookmark}

\usepackage[abbrev,msc-links,backrefs]{amsrefs} 
\usepackage{doi}

\renewcommand{\PrintDOI}[1]{\doi{#1}}

\usepackage[T1]{fontenc}
\usepackage{lmodern}
\usepackage[babel]{microtype}
\usepackage[english]{babel}

\usepackage{geometry}
\numberwithin{equation}{section}
\numberwithin{figure}{section}

\usepackage{enumitem}

\let\polishlcross=\l
\def\l{\ifmmode\ell\else\polishlcross\fi}

\def\paragraph#1{%
  \noindent\textbf{#1.}\enspace}

\let\emptyset=\varnothing
\let\setminus=\smallsetminus

\makeatletter
\def\moverlay{\mathpalette\mov@rlay}
\def\mov@rlay#1#2{\leavevmode\vtop{   \baselineskip\z@skip \lineskiplimit-\maxdimen
   \ialign{\hfil$\m@th#1##$\hfil\cr#2\crcr}}}
\newcommand{\charfusion}[3][\mathord]{
    #1{\ifx#1\mathop\vphantom{#2}\fi
        \mathpalette\mov@rlay{#2\cr#3}
      }
    \ifx#1\mathop\expandafter\displaylimits\fi}
\makeatother

\DeclareFontFamily{U}  {MnSymbolC}{}
\DeclareSymbolFont{MnSyC}         {U}  {MnSymbolC}{m}{n}
\DeclareFontShape{U}{MnSymbolC}{m}{n}{
    <-6>  MnSymbolC5
   <6-7>  MnSymbolC6
   <7-8>  MnSymbolC7
   <8-9>  MnSymbolC8
   <9-10> MnSymbolC9
  <10-12> MnSymbolC10
  <12->   MnSymbolC12}{}
\DeclareMathSymbol{\powerset}{\mathord}{MnSyC}{180}

\theoremstyle{plain}
\newtheorem{thm}{Theorem}[section]
\newtheorem{theorem}[thm]{Theorem}
\newtheorem{lemma}[thm]{Lemma}

\newtheorem{case}{Case}

\newtheorem{question}[thm]{Question}
\newtheorem*{claim*}{Claim}

\theoremstyle{definition}

\usepackage{accents}

\begin{document}

\author[K.~Heuer]{Karl Heuer}
\address{Karl Heuer, Technical University of Denmark, Department of Applied Mathematics and Computer Science, Richard Petersens Plads, Building 322, 2800 Kongens Lyngby, Denmark}
\email{\tt karheu@dtu.dk}

\title[]{Hamilton-laceable bi-powers of locally finite bipartite graphs}

\begin{abstract}
In this paper we strengthen a result due to Li by showing that the third bi-power of a locally finite connected bipartite graph that admits a perfect matching is Hamilton-laceable, i.e. any two vertices from different bipartition classes are endpoints of some common Hamilton arc.
\end{abstract}

\maketitle

\section{Introduction}\label{sec:intro}

Recently, Li~\cite{bi-power_li} restricted the power operation for graphs to preserve bipartiteness by defining another operation, called the \textit{bi-power} of a graph, which is defined as follows:

Let $G$ be a graph and $k \in \mathbb{N}$.
For two vertices $x, y \in V(G)$ let $\textnormal{dist}_G(x,y)$ denote the distance in $G$ between $x$ and $y$.
Then the \textit{$k$-th bi-power} $G_B^k$ of $G$ is defined as:
\begin{align*}
V(G_B^k) &:= V(G),\\
E(G_B^k) &:= \{ xy \; | \; \textnormal{dist}_G(x,y) \textnormal{ is odd and at most } k \textnormal{ where } x,y \in V(G) \}.
\end{align*}

Note that $G^2_B = G^1_B = G$.
Also, the $k$-th bi-power of a bipartite graph is still bipartite.

Two classical Hamiltonicity results for finite graphs that involve the (usual) square and cube of a graph are the following ones by Fleischner and Sekanina.

\begin{theorem}\cite{fleisch}\label{thm:fin_fleisch}
The square of any finite $2$-connected graph is Hamiltonian.
\end{theorem}

\begin{theorem}\cite{sekanina}\label{thm:fin-sekanina}
The cube of any finite connected graph on at least $3$ vertices is Hamiltonian.
\end{theorem}

Both of these results have been extended to \textit{locally finite} infinite graphs, i.e.~graphs where every vertex has finite degree, by Georgakopoulos~\cite{agelos}*{Thm.\ 3, Thm.\ 5}.
The crucial conceptional starting point of Georgakopoulos' work is the topological approach initiated by Diestel and K\"{u}hn~\cites{inf-cyc-1, inf-cyc-2}.
They defined infinite cycles as \textit{circles}, i.e.~homeomorphic images of the unit circle $S^1 \subseteq \mathbb{R}^2$ within the Freudenthal compactification $|G|$~\cites{diestel_buch, diestel_arx} of a locally finite connected graph $G$.
Using this notation, a \textit{Hamilton circle} of $G$ is a circle in $|G|$ containing all vertices of $G$, and we shall call $G$ \textit{Hamiltonian} if such a circle exists.

Motivated by Theorem~\ref{thm:fin_fleisch} and Theorem~\ref{thm:fin-sekanina} as well as by their extensions to locally finite infinite graphs, Li~\cite{bi-power_li} proved Hamiltonicity results for finite as well as locally finite bipartite graphs involving the bi-power.
In order to state Li's result for finite graphs we have to introduce another notation.

A finite connected bipartite graph $G$ is called \textit{Hamilton-laceable} if for any two vertices $v, w \in V(G)$ from different bipartition classes of $G$ there exists a Hamilton-path whose endvertices are $v$ and $w$.

\begin{theorem}\cite{bi-power_li}*{Thm.\ 6}\label{thm:Li-fin}
If $G$ is a finite connected bipartite graph that admits a perfect matching, then $G^3_B$ is Hamilton-laceable.
\end{theorem}

In contrast to the usual power operation and to Theorem~\ref{thm:fin_fleisch} and Theorem~\ref{thm:fin-sekanina}, the statement of Theorem~\ref{thm:Li-fin} becomes false if we omit the assumption of $G$ admitting a perfect matching as shown by Li~\cite{bi-power_li}.
We shall briefly discuss Li's example and adapt it to also yield an example for infinite graphs in Section~\ref{sec:counterexamples}.

Beside Theorem~\ref{thm:Li-fin}, Li proved a related result for locally finite infinite graphs in the same article~\cite{bi-power_li}.
However, the conclusion of the result only yields the existence of a Hamilton circle and does not speak about an adapted topological version of Hamilton-laceability.

\begin{theorem}\cite{bi-power_li}*{Thm.\ 5}\label{thm:Li-inf}
If $G$ is a locally finite infinite connected bipartite graph that admits a perfect matching, then $G^3_B$ is Hamiltonian.
\end{theorem}

In this paper, we extend Theorem~\ref{thm:Li-fin} to locally finite graphs by defining and using a natural topological extension of the notion of Hamilton-laceability.
In order to define this, we have to state other definitions first.

Let $G$ be a locally finite connected graph.
As an analogue of a path, we define an \textit{arc} as a homeomorphic image of the unit interval $[0,1] \subseteq \mathbb{R}$ in $|G|$.
We call a point $p$ of $|G|$ an \textit{endpoint} of $\alpha$ if $0$ or $1$ is mapped to $p$ by the homeomorphism defining~$\alpha$.
For $p, q \in |G|$, we shall briefly call an arc $\alpha$ a \textit{$p$--$q$ arc} if $p$ and $q$ are endpoints of $\alpha$.
Furthermore, an arc $\alpha$ in $|G|$ is called a \textit{Hamilton arc} of $G$ if it contains all vertices of $G$.

Now we are able to state a topological analogue of Hamilton-laceability.
We call a locally finite connected bipartite graph $G$ \textit{Hamilton-laceable} if for any two vertices $v, w \in V(G)$ from different bipartition classes of $G$ there exists a Hamilton arc whose endpoints are $v$ and $w$, i.e. a Hamilton $v$--$w$ arc.

Now we are able to state the main result of this paper.

\begin{theorem}\label{main}
If $G$ is a locally finite connected bipartite graph that admits a perfect matching, then $G^3_B$ is Hamilton-laceable.
\end{theorem}

Clearly, as for finite graphs, our topological notion of Hamilton-laceability is stronger than Hamiltonicity since demanding the existence of a Hamilton arc for the two endvertices of an edge in the considered graph immediately yields a Hamilton circle (unless $G \neq K_2 $).
Hence, Theorem~\ref{main} is a proper extension of Theorem~\ref{thm:Li-inf}.

The structure of this paper is as follows.
In Section~\ref{sec:prelim} we introduce the necessary notation, definitions and tools for the rest of the paper.
In Section~\ref{sec:counterexamples} we briefly discuss via finite and locally finite counterexamples how some potential strengthenings of the statement of Theorem~\ref{main} fail.
Section~\ref{sec:main-proof} starts with a brief discussion of the differences of our proof method compared to the one used by Li~\cite{bi-power_li}.
Afterwards, we prove a key lemma, Lemma~\ref{lem:key-lemma}, which then enables us to prove the main result, Theorem~\ref{main}.

\section{Preliminaries}\label{sec:prelim}

All graphs in this paper are simple and undirected ones.
Generally, we follow the graph theoretical notation from \cite{diestel_buch}.
Regarding topological notions for locally finite graphs, we especially refer to~{\cite{diestel_buch}*{Ch.\ 8.5}}, and for a wider survey about topological infinite graph theory we refer to~\cite{diestel_arx}. 

Throughout this section let $G$ denote an arbitrary, hence also potentially infinite, graph.

\subsection{Basic notions and tools}\label{subsec:basic}

For any positive integer $k$ let $[k] := \{ 1, \ldots, k \}$.

Let $X$ be a vertex set of $G$.
We denote by $G[X]$ the induced subgraph of $G$ with vertex set $X$ and write $G-X$ for the graph $G[V(G) \setminus X]$.
If $H$ is a subgraph of $G$ we shall write $G-H$ instead of $G-V(H)$.
For an edge set $E \subseteq E(G)$ we denote by $G-E$ the subgraph of $G$ with vertex set $V(G-E) := V(G)$ and edge set $E(G-E) := E(G) \setminus E$.
To ease notation in case $E$ is a singleton set, i.e. $E = \{ e \}$ for some edge $e \in E(G)$, we shall write $G - e$ instead of $G - \{e \}$.

If $T$ is a spanning tree of $G$ and $e=xy \in E(T)$, let us denote by $T_x$ and $T_y$ the two components of $T-e$ containing $x$ or $y$, respectively.
Now $D_e := E(V(T_x), V(T_y)) \subseteq E(G)$ defines a cut of $G$ and we denote it as the \textit{fundamental cut of $e$ w.r.t.~$T$} in $G$.
Also we call a cut of $G$ a \textit{fundamental cut w.r.t.~$T$} if it is a fundamental cut of some edge $e \in E(T)$ w.r.t.~$T$ in $G$.

A path $P$ is called an \textit{$X$--path} if its endvertices lie in $X$, but the set of interior vertices of $P$ is disjoint from $X$.
Similarly, for a subgraph $H \subseteq G$ we call a path a \textit{$H$--path} if it is a $V(H)$--path.
Given two vertex sets $A, B \subseteq V(G)$, we call a path $Q$ in $G$ an $A$--$B$ path if $Q$ is an $a$--$b$ path for some $a \in A$ and some $b \in B$ whose set of interior vertices is disjoint from $A \cup B$.
As before, given two subgraphs $H_1, H_2$ of $G$, we shall call a path a $H_1$--$H_2$ path if it is a $V(H_1)$--$V(H_2)$ path.
If $u$ and $v$ are vertices of a (potentially infinite) tree $T$, then we write $uTv$ to denote the unique $u$--$v$ path in $T$.

We call a one-way infinite path $R$ in $G$ a \textit{ray} of $G$, and a subgraph of $R$ that is itself a ray a \emph{tail} of $R$.
We define an equivalence relation on the set of all rays of $G$ by calling two rays in $G$ \emph{equivalent} if they cannot be separated in $G$ via any finite vertex set of $G$.
It is straightforward to check that this actually defines an equivalence relation.
For two rays $R_1, R_2$ in $G$ we shall write $R_1 \sim_G R_2$ to denote that $R_1$ and $R_2$ are equivalent in $G$.
We shall drop the subscript in case it is clear in which surrounding graph we are arguing.
Note that the statement $R_1 \sim_G R_2$ is equivalent to saying that there exist infinitely many pairwise disjoint $R_1$--$R_2$ paths in $G$.
We call the corresponding equivalence classes of rays under this relation the \textit{ends} of $G$.

A subgraph $H$ of $G$ is called \textit{end-faithful} if the following two properties hold:
\begin{enumerate}
\item [(i)] every end of $G$ contains a ray of $H$;
\item [(ii)] any two rays of $H$ belong to a common end of $H$ if and only if they belong to a common end of $G$.
\end{enumerate}

A (possibly infinite) rooted tree $T$ within a graph $G$ is called \textit{normal} if the endvertices of every $T$--path of $G$ are comparable in the tree-order of $T$.
Note that in the case of $T$ being a spanning tree, every $T$-path is just an edge.

The following theorem is due to Jung.
Since the reference~\cite{Jung_countable} is a paper written in German, we include another textbook reference for the proof of the theorem.

\begin{theorem}\cites{Jung_countable, diestel_buch}\label{thm:jung_NST}
Every countable connected graph has a normal spanning tree.
\end{theorem}

The following lemma has also been proved by Jung~\cite{Jung_countable} written in German.
As before we include an additional textbook reference for the proof.

\begin{lemma}\cites{Jung_countable, diestel_buch}\label{lem:normal_end-faithful}
Every normal spanning tree of a graph $G$ is an end-faithful subgraph of $G$.
\end{lemma}

The following lemma is a basic tool in infinite combinatorics and well-known under the name Star-Comb Lemma.
In order to formulate it we need to state another definition first.

We define a \textit{comb} as the union of a ray $R$ with infinitely many disjoint finite paths each having precisely its first vertex on $R$.
The ray $R$ is called the \textit{spine} of the comb and the last vertices of the paths are called the \textit{teeth} of the comb.

\begin{lemma}\cite{diestel_buch}*{Lemma~8.2.2}\label{lem:star-comb}
Let $U$ be an infinite set of vertices in a connected graph $G$.
Then $G$ contains either a comb with all teeth in $U$ or a subdivision of an infinite star with all leaves in $U$.
\end{lemma}

The following lemma is an easy consequence of the Star-Comb Lemma and should be known.
We include the proof for the sake of completeness.

\begin{lemma}\label{lem:end-faith_fund_cut}
Let $G$ be a locally finite connected graph and let $T$ be an end-faithful spanning tree of $G$.
Then every fundamental cut w.r.t.~$T$ in $G$ is finite.
\end{lemma}

\begin{proof}
Let $e =xy \in E(T)$ and $D_e \subseteq E(G)$ be the fundamental cut of $e$ w.r.t.~$T$ in $G$.
Suppose for a contradiction that $D_e$ is infinite.
Now we apply Lemma~\ref{lem:star-comb} to the vertex set $X = \left(\bigcup D_e \right) \cap V(T_x)$ in $T_x$ and to the vertex set $Y = \left(\bigcup D_e \right) \cap V(T_y)$ in $T_y$.
Since $G$ is locally finite, the application of Lemma~\ref{lem:star-comb} yields combs $C_x$ and $C_y$ in $T_x$ and $T_y$, respectively, where the teeth of $C_x$ lie in $X$ and the teeth of $C_y$ lie in $Y$.
Let $S_x$ and $S_y$ denote the spines of $C_x$ and $C_y$, respectively.
Now within the graph $C_x \cup C_y \cup G[\bigcup D_e] \subseteq G$ there exist infinitely many disjoint $S_x$--$S_y$ paths, witnessing  that $S_x \sim_G S_y$.
As $S_x$ and $S_y$ are contained in $T$ but $S_x \nsim_{T} S_y$, we have derived a contradiction to $T$ being an end-faithful spanning tree of $G$.
\end{proof}

\subsection{Topological notions and tools}\label{subsec:top}

For this subsection we assume $G$ to be a locally finite connected graph.
We can endow the $1$-skeleton of $G$ together with its ends with a certain topology, yielding the space $|G|$ referred to as \emph{Freudenthal compactification} of $G$.
For a precise definition of~$|G|$, see \cite{diestel_buch}*{Ch.\ 8.5}.
Furthermore, we refer to~\cite{Freud} for Freudenthal's paper about the Freudenthal compactification, and to~\cite{Freud-Equi} regarding the connection to $|G|$.
Note that the definition of $|G|$ ensures that each edge of $G$ corresponds to an individual copy of the real unit interval $[0, 1]$ within $|G|$ and for adjacent edges of~$G$, appropriate endpoints of the corresponding unit intervals are identified.

We denote the \textit{closure} of a point set $X \subseteq |G|$ in $|G|$ by $\overline{X}$.
A subspace $S$ of $|G|$ is called a \textit{standard subspace} if $S = \overline{F}$ for some edge set $F \subseteq E(G)$.

The next lemma yields an important combinatorial property of arcs.
In order to state the lemma, let $\mathring{F}$ denote the set of inner points of edges $e \in F$ in $|G|$ for an edge set $F \subseteq E(G)$.

\begin{lemma}{\cite{diestel_buch}*{Lemma 8.5.3}}\label{lem:jumping-arc}
Let $G$ be a locally finite connected graph and ${F \subseteq E(G)}$ be a cut with sides $V_1$ and $V_2$.
If $F$ is finite, then $\overline{V_1} \cap \overline{V_2} = \emptyset$, and there is no arc in $|G| \setminus \mathring{F}$ with one endpoint in $V_1$ and the other in $V_2$.
\end{lemma}

The following lemma ensures that being connected or being arc-connected are equivalent for closed subspaces of $|G|$.

\begin{lemma}{\cite{path-cyc-tree}*{Thm.\ 2.6}}\label{lem:arc_conn}
If $G$ is a locally finite connected graph, then every closed topologically connected subset of $|G|$ is arc-connected.
\end{lemma}

The next lemma characterises the property of a standard subspace of being topologically connected, and due to Lemma~\ref{lem:arc_conn} also being arc-connected, in terms of a purely combinatorial condition, which we shall make use of later.

\begin{lemma}{\cite{diestel_buch}*{Lemma 8.5.5}}\label{lem:top_conn}
If $G$ is a locally finite connected graph, then a standard subspace of $|G|$ is topologically connected (equivalently: arc-connected) if and only if it contains an edge from every finite cut of $G$ of which it meets both sides.
\end{lemma}

\section{Counterexamples for potential strengthenings of Theorem~\ref{main}}\label{sec:counterexamples}

\subsection{No perfect matching in \texorpdfstring{$G$}{G}}\label{subsec:matching}

In this first subsection we discuss why the statement of Theorem~\ref{main} becomes false if we omit the assumption of $G$ having a perfect matching, even if we additionally assume higher connectivity of the graph $G$ and focus on higher bi-powers.
For finite graphs, Li~\cite{bi-power_li} gave an example by constructing for every $k, \ell \in \mathbb{N}$ a $k$-connected balanced bipartite graph $G$ that does not admit a perfect matching and where $G^{\ell}_B$ is not Hamiltonian.

Let us now recall Li's example and afterwards slightly extend it to also yield an example for locally finite infinite graphs, of course except for the property of being balanced.
Let $s \geq \ell$ be an even number.
Fix disjoint vertex sets $V_0, \ldots, V_{s+1}$ where $|V_0| = |V_{s+1}| >  sk/2$ and $|V_i| = k$ for every $i \in [s]$.
Define the graph $L_{k, s}$ by setting $V(L_{k, s}) = \bigcup^{s+1}_{i=0} V_i$ and by adding all possible edges between $V_i$ and $V_{i+1}$ for all $i \in \{ 0, 1, \ldots, s+1 \}$.

Clearly, $L_{k, s}$ is a finite balanced $k$-connected bipartite graph.
By construction, the vertex set $V_0 \cup V_{s+1}$ forms an independent set of size greater than $|V(L_{k, s})|/2$.
Hence, neither does $L_{k, s}$ admit a perfect matching nor $(L_{k, s})^{\ell}_B$ a Hamilton cycle; consequently $(L_{k, s})^{\ell}_B$ cannot be Hamilton-laceable either.

For locally finite infinite graphs let us analogously define the graph $H_{k, \ell}$ as follows.
Let $V_i$ denote disjoint vertex sets for every $i \in \mathbb{N}$ where $|V_0| > \left \lfloor \frac{\ell}{2} \right \rfloor k$ and $|V_i| = k$ for every $i > 0$.
Similarly as before, set $V(H_{k, \ell}) = \bigcup_{i \in \mathbb{N}} V_i$ and add all edges between $V_i$ and $V_{i+1}$ for all $i \in \mathbb{N}$.

Clearly, $H_{k, \ell}$ is a locally finite $k$-connected bipartite graph.
Since $V_0$ is an independent set and $|V_0| > |N(V_0)| = |V_1|$, we also have that $H_{k, \ell}$ does not contain a perfect matching for any $\ell \in \mathbb{N}$.
Furthermore, in $(H_{k, \ell})^{\ell}_B$ the set $V_0$ has $\left \lfloor \frac{\ell}{2} \right \rfloor k < |V_0|$ many neighbours.
Hence, $(H_{k, \ell})^{\ell}_B$ cannot be Hamiltonian, and hence also not be Hamilton-laceable.

\subsection{Perfect matching in \texorpdfstring{$G^3_B$}{the bi-cube of G} instead of \texorpdfstring{$G$}{G}}

In~\cite{bi-power_li} Li also briefly addresses the question whether the assumption within the statement of Theorem~\ref{thm:Li-fin} of $G$ admitting a perfect matching might be weakened to $G^3_B$ admitting a perfect matching.
As Hamiltonicity for bipartite graphs implies the existence of a perfect matching, this is a necessary condition.
However, Li gave a counterexample for this condition to also be sufficient for finite graphs in the context of Theorem~\ref{thm:Li-fin}.
We shall now briefly recall Li's example and then slightly extend it to also yield a counterexample with respect to locally finite infinite graphs and Theorem~\ref{main}.

For $k \geq 3$ start with a tree that has precisely two vertices of degree $k+1$ while all other vertices have degree $1$.
Next subdivide each edge that is incident with a leaf twice and call the resulting graph $L_k$.
As Li noted, it is easy to check that $(L_k)^3_B$ does admit a perfect matching, but no Hamilton cycle.

For a locally finite infinite graph we now state the following construction.
Start with a star $K_{1, k}$ whose centre is $c$ and that has precisely $k$ leaves $\ell_1, \ldots, \ell_k$ for some $k \geq 3$.
Next subdivide each edge precisely twice.
Now take the disjoint union of the resulting graph with a ray $R = r_1r_2 \ldots$.
Finally, add the edge $cr_1$ and $k$ further vertices that are all only adjacent to $r_1$.
Let us call the resulting graph $H_k$.

As in Li's example, it is easy to check that $(H_k)^3_B$ admits a perfect matching.
The key observation why $(H_k)^3_B$ is not Hamiltonian is also the same as in Li's example, which we briefly recall with respect to our example.
Note that the vertices $\ell_1, \ldots, \ell_k$ have degree $2$ in $(H_k)^3_B$ and they all share $c$ as a common neighbour.
Hence, any Hamilton circle $\overline{C}$ of $(H_k)^3_B$, where $C$ denotes some subgraph of $(H_k)^3_B$, would impose $c$ to have degree $k \geq 3$ in~$C$, which is not possible.

\subsection{Hamilton-connectedness}\label{subsec:hamilton-con}

Another question that might arise when considering the statements of Theorem~\ref{thm:Li-fin} and Theorem~\ref{main} is whether even Hamilton-connectedness can be deduced.
To recall: the notion of \textit{Hamilton-connectedness} is analogously defined to the one of Hamilton-laceability for finite as well as for locally finite graphs but by considering all pairs of distinct vertices instead of just those that have odd distance from each other.
For finite bipartite graphs and the statement of Theorem~\ref{thm:Li-fin} we can clearly not deduce Hamilton-connectedness as the assumption of having a perfect matching ensures our graph to be balanced, while a Hamilton path with endvertices in the same bipartition class would violate this.

For locally finite infinite bipartite graphs we can also negatively answer the question via a counterexample very easily.
Consider a ray $R = r_1r_2 \ldots$.
The graph $R^3_B$ does not admit a Hamilton $r_2$--$r_{2k}$ arc for any $k > 1$.
To see this, note first that $r_1$ has degree $2$ in~$R^3_B$.
Hence, any potential spanning arc starting in $r_2$ would have the path $r_2r_1r_4$ as initial segment.
Now $r_3$ has degree $2$ in $R^3_B - r_2$, which forces the initial segment $r_2r_1r_4r_3r_6$.
By iterating this argument, we see that any potential spanning arc for $R^3_B$ starting in $r_2$ would contain an $r_2$--$r_{2k}$ path as an initial segment, preventing the existence of a Hamilton $r_2$--$r_{2k}$ arc.

\section{Proof of the main result}\label{sec:main-proof}

We start this section by very briefly sketching and discussing the rough methodological differences regarding how Li proved Theorem~\ref{thm:Li-inf} in~\cite{bi-power_li} and how we prove Theorem~\ref{main}.
Before we can relate the different approaches, we have to give some additional definitions first.

Let $G$ be a locally finite infinite graph and $\omega$ be an end of $G$.
We define the \textit{degree} of $\omega$ in $G$ to be the supremum of the number of vertex-disjoint rays in $G$ which are contained in $\omega$.
Furthermore, we call a continuous image of $S^1 \subseteq \mathbb{R}^2$ within $|G|$ that contains all vertices of $G$ a \textit{Hamilton curve} of $G$.
Note that in contrast to a Hamilton circle of $G$, a Hamilton curve of $G$ may traverse an end of $G$ several times.

Li's approach to verify Hamiltonicity of $G^3_B$ is to first start with an end-faithful spanning tree $T$ of $G$ which contains a perfect matching of $G$.
Then he carefully extends $T$ by edge sets of suitable cycles such that an end-faithful spanning subgraph $G'$ of $G^3_B$ is obtained that admits a Hamilton curve and each of whose ends has degree at most $3$.
It is easy to see that every Hamilton curve in $G'$ is actually already a Hamilton circle as the end degrees do not allow to traverse an end multiple times.
Due to the end-faithfulness, that Hamilton circle of $G'$ is also one of $G^3_B$.
Furthermore, Li chooses the set of suitable cycles in such a way that they also prove the existence of a Hamilton curve of $G'$.
For this he makes use of the following characterisation for the existence of Hamilton curves.

\begin{theorem}\cite{hamilton_curves}\label{thm:H-curve}
A locally finite connected graph $G$ has a Hamilton curve if and only if every finite vertex set of $G$ is contained in some finite cycle of $G$.
\end{theorem}

The way how Li precisely constructs the mentioned set of cycles goes back to his proof of Lemma~1 in~\cite{bi-power_li}, which is basically the finite version of Lemma~\ref{lem:key-lemma}.

Although using Theorem~\ref{thm:H-curve} can be a powerful and convenient tool, it does not seem immediately helpful for the purpose of verifying Hamilton-laceability (or Hamilton-connectedness).
Our way to prove Theorem~\ref{main} is to mimic Li's proof of Theorem~\ref{thm:Li-fin} in finite graphs.
Hence, we especially construct certain Hamilton arcs of the considered graph $G$ directly within the third bi-power of an end-faithful spanning tree of $G$ that contains a perfect matching of $G$.
This part of our proof happens in Lemma~\ref{lem:key-lemma}.
Later in the proof of Theorem~\ref{main} when we apply Lemma~\ref{lem:key-lemma}, we combine Hamilton arcs of suitable subgraphs of $G^3_B$ along an inductive argument to yield the desired Hamilton arcs of $G^3_B$.
The general idea of this part is the same as in Li's proof of Theorem~\ref{thm:Li-fin}.
However, we have to build our induction on a different parameter, namely on distances between vertices, to ensure that the parameter is always finite.
Similarly as Lemma~1 in~\cite{bi-power_li} was the key lemma in Li's proof of Theorem~\ref{thm:Li-fin}, now Lemma~\ref{lem:key-lemma} is our key lemma to prove Theorem~\ref{main}.

Now let us start preparing to prove Theorem~\ref{main}.
The following lemma ensures that we can always extend a perfect matching of a countable graph $G$ to an end-faithful spanning tree of $G$.
Although this lemma can be deduced from a more general lemma in Li's article~\cite{bi-power_li}*{Lemma~6}, we decided to include a proof here for the sake of keeping this article self-contained and because our proof seems simpler due to the less technical setting.

\begin{lemma}\label{lem:matching_in_tree}
Let $G$ be a countable connected graph and $M$ be a perfect matching of $G$.
Then there exists an end-faithful spanning tree of $G$ that contains $M$.
\end{lemma}

\begin{proof}
Let $G$ and $M$ be as in the statement of the theorem.
Now we apply Theorem~\ref{thm:jung_NST} with the graph $G/M$, which is still a countable connected graph, guaranteeing us the existence of a normal spanning tree $T'$ of $G/M$.
Next we uncontract every edge of $M$ in $T'$ and form, within $G$, a spanning tree $T$ of $G$.
Note for this that $M$ is a perfect matching of~$G$.
Hence, every vertex of $T'$ corresponds to an edge $m \in M$, and so we shall also work with $M$ as the vertex set of $T'$ for the rest of the proof.
In order to define $T$ we pick for every edge $e_{T'} = m_1m_2 \in E(T')$ an arbitrary edge $e_{T} \in E(G)$ that witnesses the existence of $m_1m_2 \in E(T')$, i.e. $e_{T} = u_1u_2$ where $u_1, u_2 \in V(G)$ and $u_i$ is an endvertex of $m_i$ for each $i \in \{ 1, 2 \}$.
Now we define $T$ as follows:
\begin{align*}
V(T) &:= V(G),\\
E(T) &:= M \cup \{ e_{T} \; | \; e_{T'} \in E(T') \}.
\end{align*}
Obviously, $T$ is still connected and does not contain a finite cycle.
Hence, $T$ is a spanning tree of $G$ that contains $M$ by definition.

Next we verify that $T$ is an end-faithful subgraph of $G$.
Let $\omega$ be any end of $G$ and $R \in \omega$.
Now $R/M$ is a ray in $G/M$, and since $T'$ is an end-faithful spanning tree of $G/M$, there exists a ray $R'$ in $T'$ such that $R' \sim_{G/M} R/M$.
Let $\mathcal{P}'$ be a set of infinitely many pairwise disjoint $R'$--$R/M$ paths in $G/M$ witnessing $R' \sim_{G/M} R/M$.
Now let $R_T$ be any ray in $T$ obtained from $T \left[\bigcup \{ e_T \; | \; e_{T'} \in E(R') \}\right]$ by adding edges from $M$.
By uncontracting edges from $M$, the path system $\mathcal{P}'$ now gives rise to a set of infinitely many pairwise disjoint $R_T$--$R$ paths in $G$. Hence, $R_T \subseteq T$ is a desired ray satisfying $R_T \sim_G R$.

Now let $R_1, R_2$ be two rays of $T$.
If $R_1 \sim_T R_2$, then $R_1$ must be a tail of $R_2$ or vice versa since $T$ is a tree.
Hence, $R_1 \sim_G R_2$.
Conversely, suppose for a contradiction that $R_1 \sim_G R_2$ but $R_1 \nsim_T R_2$.
First, note that $R_1/M \sim_{G/M} R_2/M$ holds.
Using that $T'$ is an end-faithful subgraph of $G/M$ we also know $R_1/M \sim_{T'} R_2/M$.
However, from $R_1 \nsim_T R_2$ we know that a finite set $S \subseteq V(G)$ exists separating $R_1$ and $R_2$ in $T$.
Now the set $S' := \{ m \in M \; | \; m \cap S \neq \emptyset \}$ defines a vertex set in $T'$ that separates $R_1/M$ and $R_2/M$ in~$T'$; a contradiction.
\end{proof}

The following question arose while preparing this article.
It basically asks whether we can also get a normal spanning tree to satisfy the conclusion of Lemma~\ref{lem:matching_in_tree} instead of just an end-faithful one.
Although neither a positive nor a negative answer to this question would substantially affect or shorten the proof of the main result of this paper, the question seems to be of its own in interest.
Hence, it is included here.

\begin{question}\label{question:NST-match-count}
Let $G$ be a countable connected graph and let $M$ be a perfect matching of~$G$.
Does $G$ admit a normal spanning tree that contains $M$?
\end{question}

Since the assumption of countability in Lemma~\ref{lem:matching_in_tree} was only used to ensure the existence of a normal spanning tree, the following question is a related, but more general one than Question~\ref{question:NST-match-count}.

\begin{question}\label{question:NST-match}
Let $G$ be a connected graph and $M$ be a perfect matching of~$G$ such that $G/M$ admits a normal spanning tree.
Does $G$ admit a normal spanning tree that contains~$M$?
\end{question}


Note that both questions above have positive answers when we restrict them to finite graphs.
We can easily include the desired perfect matching during the constructing of a depth-first search spanning tree, which in particular is a normal one.

During a discussion with Carsten Thomassen it turned out that Question~\ref{question:NST-match-count} has an easy counterexample, which then also negatively answers Question~\ref{question:NST-match}.
In the following lines we shall state the counterexample.
We shall follow the convention that the set of natural numbers $\mathbb{N}$ contains the number $0$.

For $i = 1, 2$ let $V^i = \{ v^i_0, v^i_1, \ldots \}$ be two disjoint countably infinite vertex sets.
Now we define our desired graph $G$ as follows.
Let $V(G) := V^1 \cup V^2$.
Furthermore, for each $i = 1, 2$ we define edge sets $E^i := \{ v^i_jv^i_{j+1} \; | \; j \in \mathbb{N} \}$.
Finally, we define the edge set $E^* = \{ v^1_{2k}v^2_{2k} \; | \; k \in \mathbb{N} \setminus \{ 0 \} \}$, and set $E(G) := E^1 \cup E^2 \cup E^*$.
Next we define a perfect matching $M$ of $G$ which cannot be included in any normal spanning tree of $G$, independent of the choice of the root vertex for the tree.
Set $M := \{ v^i_{2k}v^i_{2k+1} \; | \; i \in \{ 1, 2\} \, , \, k \in \mathbb{N} \}$.
See Figure~\ref{counterexample} for a picture of the graph $G$ together with its perfect matching $M$.

\begin{figure}[htbp]
\centering
\includegraphics[width=9cm]{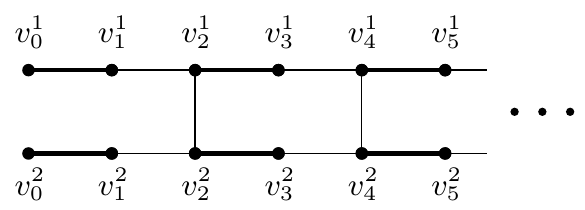}
\caption{The graph $G$ with the perfect matching $M$ indicated by bold edges.}
\label{counterexample}
\end{figure}

Next we prove why no normal spanning tree of $G$ can contain~$M$, independent of the choice for the root of the tree.
Note first that any normal spanning tree $T$ of $G$ must contain infinitely many edges from~$E^*$.
This is because otherwise $T$ would contain two disjoint rays in the end of $G$, which contradicts the fact that normal spanning trees are end-faithful and, therefore, contain only a unique ray in each end of $G$.
See \cite{Jung_countable} for a proof of this fact.
Now suppose for a contradiction that $T$ is a normal spanning tree of $G$ containing $M$.
Let $r \in V(T)$ denote the root of $T$, and let $j \in \{ 1, 2 \}$ and $\ell \in \mathbb{N}$ be such that $r = v^j_{\ell}$.
Next let $p \in \mathbb{N}$ be the smallest index for which $p > \ell$ and $v^1_pv^2_p \in E(T)$ hold.
Similarly, let $q \in \mathbb{N}$ be the second smallest index for which $q > \ell$ and $v^1_qv^2_q \in E(T)$ hold.
Clearly, the paths $v^1_p \ldots v^1_{q}$ and $v^2_p \ldots v^2_{q}$ cannot both be contained in $T$ since then $T$ would contain a finite cycle.
Without loss of generality say that $v^1_p \ldots v^1_{q}$ is not contained in $T$.
Then let $s \in \mathbb{N}$ be the smallest index such that $p < s < q$ and $v^1_sv^1_{s+1} \notin E(T)$.
Since $M$ is contained in $T$, we know that $s$ is an odd number.
So $v^1_s$ has degree $2$ in $G$ and is a leaf in $T$.
Since $\ell < s$, the vertices $v^1_s$ and $v^1_{s+1}$ lie on different branches in $T$.
Hence, the edge $v^1_sv^1_{s+1}$ contradicts the normality of $T$.
Therefore, we can conclude that it is impossible for any normal spanning tree of $G$ to contain the perfect matching $M$.


Now we continue with the key lemma for the proof of our main result.

\begin{lemma}\label{lem:key-lemma}
Let $G$ be a locally finite connected bipartite graph and $M$ be a perfect matching of $G$.
Furthermore, let $T$ be an end-faithful spanning tree of $G$ containing $M$.
Then for every edge $xy \in M$ there exists a Hamilton $x$--$y$ arc of $G^3_B$ within $\overline{T^3_B} \subseteq |G^3_B|$
\end{lemma}

\begin{proof}
Let $G$, $T$ and $M$ be as in the statement of the lemma.
First note that, as $G$ is bipartite, clearly $T^3_B \subseteq G^3_B$.
Now let $m^1_0 = x^1_0y^1_0$ be an arbitrary edge from $M$.
We recursively make the following definitions:
Set $T_0 = T[\{x^1_0, y^1_0\}]$.
Now for every $i \in \mathbb{N}$, set $T_{i+1}$ to be the subtree of $T$ induced by all vertices of $T$ which are contained in some edge $m' \in M$ such that $m'$ has one endvertex in distance at most $1$ to $T_i$ within $T$.
Next we shall recursively define a Hamilton $x^1_0$--$y^1_0$ path $A_i$ in each $(T_i)^3_B$ for every $i \in \mathbb{N}$ where $A_i$ contains each edge $m \in M$ which is contained in $T_i - V(T_{i-1})$, where $V(T_{-1}) := \emptyset$.

First set $E(A_0) := m^1_0$.
Now suppose we have already defined $A_i$ and want to define~$A_{i+1}$.
Let us enumerate the edges in $T_{i} - V(T_{i-1})$ that are contained in $E(A_{i}) \cap M$ by $m^1_{i}, \ldots, m^{p_i}_{i}$ for some $p_i \in \mathbb{N}$.
Furthermore, let us fix names for the endvertices of each such edge by writing $m^j_{i} = x^j_iy^j_i$ for every $j \in [p_i]$.
By definition, either $x^j_i$ or $y^j_i$ is adjacent to some endvertex $y^{q}_{i-1}$, for some $q \in [p_{i-1}]$, of an edge $m^q_{i-1} \in M$ contained in $T_{i-1}$, w.l.o.g.~say $x^j_i$.
Let $T_{x^j_i}$ denote the component of $T_{i+1} - y^q_{i-1}x^j_i$ containing $x^j_i$ if $i > 0$, and set $T_{x^1_0} := T_1$ for $i = 0$.
Now we shall extend $A_i$ further into each $T_{x^j_i}$ (unless $E(T_{x^j_i}) = m^j_i$) by replacing the edges $m^j_i$ for all $j \in [p_i]$ in $A_i$ by a Hamilton $x^j_i$--$y^j_i$ path $P^j_{i}$ of $(T_{x^j_i})^3_B$ to form $A_{i+1}$.
We shall distinguish three cases of how to do this for each $j \in [p_i]$.

\begin{case}
Both components of $T_{x^j_i}-m^j_i$ are trivial.
\end{case}

In this case we keep the edge $m^j_i$ to also be an edge of $A_{i+1}$.

\begin{case}
Precisely one component of $T_{x^j_i}-m^j_i$ is trivial, w.l.o.g.~say the one containing $x^j_i$.
\end{case}

In this situation let $N(y^j_i) = \{ x^j_i, v_1, v_2, \ldots, v_k\}$ for some $k \geq 1$.
Since $M$ is a perfect matching of $G$, each $v_r$ must be contained in some edge from $M$, say $v_rw_r \in M$ for every $r \in [k]$.
Now set $P^j_i := x^j_iw_1v_1w_2v_2 \ldots w_kv_ky^j_i$.

\begin{case}
No component of $T_{x^j_i}-m^j_i$ is trivial.
\end{case}

Now let ${N(y^j_i) = \{ x^j_i, v_1, v_2, \ldots, v_k\}}$ for some $k \geq 1$ and ${N(x^j_i) = \{ y^q_{i-1}, y^j_i, a_1, a_2, \ldots, a_{\ell}\}}$ for some $\ell \geq 1$.
Using as before that $M$ is a perfect matching of $G$, each $v_r$ and each $a_s$ must be contained in some edge of $M$, say $v_rw_r \in M$ for every $r \in [k]$ and $a_sb_s \in M$ for every $s \in [\ell]$.
Finally, set $P^j_i := x^j_iw_1v_1w_2v_2 \ldots w_kv_ka_1b_1a_2b_2 \ldots a_kb_ky^j_i$.\\

This completes the definition of $A_{i+1}$.
Note that this definition ensures that an edge $e \in A_i$ is also contained in $A_{i+1}$ except $e = m^j_i \in M$ for some $j \in [p_i]$ as above and at least one component of $T_{x^j_i}-m^j_i$ is non-trivial.

Next we define our desired Hamilton $x^1_0$--$y^1_0$ arc in $\overline{T^3_B} \subseteq |G^3_B|$.
For this we set
\[ A:= \left \lbrace e \in \bigcup_{n \in \mathbb{N}} E(A_n) \; | \; e \textnormal{ is contained in all but finitely many } A_n \right \rbrace. \]
Now we claim that $\overline{A}$ is a desired $x^1_0$--$y^1_0$ Hamilton arc of $G^3_B$.

Note first that by definition $\overline{A} \subseteq \overline{T^3_B} \subseteq |G^3_B|$ and $\overline{A}$ contains all vertices of $T$, and therefore all vertices of $G$ because $T$ is a spanning tree of $G$.
It remains to verify that $\overline{A}$ is an $x^1_0$--$y^1_0$ arc.
We shall do this by showing that $\overline{A}$ contains an $x^1_0$--$y^1_0$ arc but $\overline{A-e}$ does not for any edge $e \in A$, which implies that $\overline{A}$ is an $x^1_0$--$y^1_0$ arc.

To prove that $\overline{A}$ contains an $x^1_0$--$y^1_0$ arc in $|G^3_B|$ it is enough to show that $A$ intersects every finite cut of $G^3_B$ due to Lemma~\ref{lem:top_conn}.
Now let $F \subseteq E(G^3_B)$ be an arbitrary finite cut of~$G^3_B$.
Hence, $F \subseteq E(G^3_B[V(T_n)])$ for some $n \in \mathbb{N}$.
Since $A_{n+1}$ is a Hamilton $x^1_0$--$y^1_0$ path in $(T_{n+1})^3_B$ and $E(A_{n+1}) \cap E(G^3_B[V(T_n)]) = A \cap E(G^3_B[V(T_n)])$, we know that $A$ intersects~$F$.

For the remaining argument that $\overline{A-e}$ does not contain an $x^1_0$--$y^1_0$ arc for any edge $e \in A$, we shall first find a finite cut $F$ of $G^3_B$ that $A$ intersects precisely in $e$.
Let $n \in \mathbb{N}$ such that $e \in E(G^3_B[V(T_n)])$.
Since $A_n$ is a Hamilton $x^1_0$--$y^1_0$ path in $G^3_B[V(T_n)]$, there exists a cut $F_n = E(L_n, R_n)$ of $G^3_B[V(T_n)]$ such that $x^1_0 \in L_n$ and $y^1_0 \in R_n$ and $E(A_n) \cap F_n = \{ e \}$.
Let $\mathcal{C}_L$ and $\mathcal{C}_R$ be the sets of all components of $T-E(T_n)$ that intersect $L_n$ and $R_n$, respectively.
Next we extend the bipartition $(L_n, R_n)$ of $V(T_n)$ to a bipartition $(L, R)$ of~$V(T)$.
We set $L := \bigcup \{ V(C) \; | \; C \in \mathcal{C}_L \}$ and $R := \bigcup \{ V(C) \; | \; C \in \mathcal{C}_R \}$.
Especially, this yields $x^1_0 \in L_n \subseteq L$ and $y^1_0 \in R_n \subseteq R$.
Furthermore, $T$ intersects the cut $F_G := E(L, R)$ of $G$ in the same edges as $T_n$ intersects $F_n$.
Hence, $T$ intersects $F_G$ in only finitely many edges.
Next note that any edge $f = uv \in F_G \setminus E(T)$ lies in the fundamental cut $D_{g}$ of $G$ w.r.t. $T$ for every edge $g$ that lies on the $u$--$v$ path in $T$.
Especially, $f$ lies in $D_{g'}$ for some of the finitely many edges $g' \in F_G \cap E(T)$.
As every fundamental cut of $G$ w.r.t. $T$ is finite by Lemma~\ref{lem:end-faith_fund_cut}, this implies that $F_G$ is a finite cut of $G$.
Since $G$ is locally finite and by definition of $G^3_B$, we furthermore get that the bipartition $(L, R)$ of $V(G) = V(G^3_B)$ also yields a finite cut $F$ of~$G^3_B$.
By the definition of the $A_i$'s we know that every $A_m$ for $m \geq n$ also satisfies $E(A_m) \cap F = \{ e \}$.
Hence, $A \cap F = \{ e \}$.

To complete the argument, note that every $x^1_0$--$y^1_0$ arc in $|G^3_B|$ must intersect $F$ by Lemma~\ref{lem:jumping-arc}.
This, however, implies that $\overline{A-e}$ cannot not contain an $x^1_0$--$y^1_0$ arc.
\end{proof}

Now we are able to prove our main result.

\begin{proof}[Proof of Theorem~\ref{main}]
Let $u$ and $v$ be two vertices from different bipartition classes of $G$ and let $M$ be a perfect matching of $G$.
By Lemma~\ref{lem:matching_in_tree} there exists an end-faithful spanning tree $T$ of $G$ that contains $M$.
Since $G$ is bipartite, the distance between $u$ and $v$ in $T$ is also odd.
We now prove the statement of the theorem by induction on $d := |E(uTv) \setminus M|$.
Since $|E(uTv)|$ is odd and $|E(uTv) \cap M| \leq \left \lfloor \frac{|E(uTv)|}{2} \right \rfloor$ as $M$ is a perfect matching of $G$, we know that $d = 0$ holds precisely when $E(uTv) = \{ uv \}$ and $uv \in M$.
Now the statement follows from Lemma~\ref{lem:key-lemma}.

Next let us verify the statement for $d > 0$ while assuming we have verified it for all smaller values for $d$.
There must exist an edge $xy \in E(uTv) \setminus M$, say without loss of generality $x \in E(uTy)$.
As $T$ is end-faithful, the fundamental cut $D_{xy}$ w.r.t.~$T$ in $G$ is finite by Lemma~\ref{lem:end-faith_fund_cut}.
Note that $T$ is also an end-faithful spanning tree of $G^3_B$.
To see this observe first that given any ray $R$ in $G^3_B$ we obtain by applying Lemma~\ref{lem:star-comb} to $V(R)$ within $T \subseteq G^3_B$, and due to $T$ being locally finite, a comb whose spine is equivalent to $R$ in $T$, and hence also in $G^3_B$.
Second, let two non-equivalent rays $R_1, R_2$ in $T$ be given.
As $T$ is an end-faithful spanning tree of $G$, there exists a finite vertex set $S \subseteq V(G)$ such that $R_1 - S$ and $R_2 - S$ lie in different components of $G-S$.
By definition of $G^3_B$ and due to the locally finiteness of $G$, we get that $S \cup N(S)$ is a finite vertex set separating $R_1$ and $R_2$ in $G^3_B$.
Now since $G$ is locally finite and $T$ is also an end-faithful spanning tree of $G^3_B$, we know that the fundamental cut $D^B_{xy}$ w.r.t.~$T$ in $G^3_B$ is finite as well.
For ease of notation set $H := G[T_x]$ and $K := G[T_y]$.
Due to the finite fundamental cut $D_{xy}$, we know that the spaces $|H|$ and $|K|$ are homeomorphic to the subspaces of $|G|$ induced by the closures $\overline{H}$ and $\overline{K}$, respectively.
Furthermore, $\overline{H} \cap \overline{K} = \emptyset$ by Lemma~\ref{lem:jumping-arc}.
The same observations hold for $H^3_B$ and $K^3_B$ since $T$ is also an end-faithful spanning tree of $G^3_B$ implying, as note before, that the fundamental cut $D^B_{xy}$ w.r.t.~$T$ in $G^3_B$ is finite.
Next we shall make a case distinction of how to apply our induction hypothesis.
Note that since $|E(uTv)|$ is odd, either $|E(uTx)|$ and $|E(yTv)|$ are both odd or they are both even.

\setcounter{case}{0}
\begin{case}
$|E(uTx)|$ and $|E(yTv)|$ are odd.
\end{case}

In this case we apply our induction hypothesis with the graphs $H$ and $K$, their perfect matchings $M_x := M \cap E(T_x)$ and $M_y := M \cap E(T_y)$, the end-faithful spanning trees $T_x$ and $T_y$, which contain $M_x$ and $M_y$ respectively, and the pairs of vertices $(u, x)$ and $(y, v)$.
Hence we obtain a Hamilton $u$--$x$ arc $A_x$ of $H^3_B$ within within $\overline{(T_x)^3_B}$ and a Hamilton $y$--$v$ arc $A_y$ of $K^3_B$ within within $\overline{(T_y)^3_B}$.
Since $\overline{H^3_B} \cap \overline{K^3_B} = \emptyset$ holds within $|G^3_B|$, we obtain a Hamilton $u$--$v$ arc of $G^3_B$ by joining $A_x$ and $A_y$ via the edge $xy$.

\begin{case}
$|E(uTx)|$ and $|E(yTv)|$ are even.
\end{case}

First note for this case that there exist edges $xx', yy' \in M$ since $M$ is a perfect matching.
As $M \subseteq E(T)$ by assumption, we know that $xx' \in E(T_x)$ and $yy' \in E(T_y)$.
Also note that $|E(uTv \setminus M)| = |E(uTx' \setminus M)| + |E(y'Tv \setminus M)| + 1$, whether $x'$ or $y'$ are contained in $uTv$ or not, and that $|E(uTx')|$ and $|E(y'Tv)|$ are both odd.
Furthermore, $\textnormal{dist}_T(x', y') = 3$, and, therefore, $\textnormal{dist}_G(x', y') \in \{1, 3\}$.
Hence, $x'y' \in E(T^3_B) \subseteq E(G^3_B)$.
Due to these observations we can apply our induction hypothesis as in Case~1 but with $x'$ and $y'$ instead of $x$ and~$y$, yielding again the desired Hamilton $u$--$v$ arc of $G^3_B$.
\end{proof}

\section*{Acknowledgements}
Karl Heuer was supported by the European Research Council (ERC) under the European Union's Horizon 2020 research and innovation programme (ERC consolidator grant DISTRUCT, agreement No.\ 648527).


\begin{bibdiv}
\begin{biblist}



\bib{diestel_buch}{book}{
   author={Diestel, Reinhard},
   title={Graph theory},
   series={Graduate Texts in Mathematics},
   volume={173},
   edition={5},
   publisher={Springer, Berlin},
   date={2017},
   pages={xviii+428},
   isbn={978-3-662-53621-6},
   review={\MR{3644391}},
   doi={10.1007/978-3-662-53622-3},
}



\bib{diestel_arx}{article}{
	author={Diestel, Reinhard}, 
	title={Locally finite graphs with ends: a topological approach}, 
	date={2012},
	eprint={0912.4213v3},
	note={Post-publication manuscript},
}



\bib{path-cyc-tree}{article}{
   author={Diestel, Reinhard},
   author={K\"{u}hn, Daniela},
   title={Topological paths, cycles and spanning trees in infinite graphs},
   journal={European J. Combin.},
   volume={25},
   date={2004},
   number={6},
   pages={835--862},
   issn={0195-6698},
   review={\MR{2079902}},
   doi={10.1016/j.ejc.2003.01.002},
}


\bib{Freud-Equi}{article}{
   author={Diestel, Reinhard},
   author={K\"{u}hn, Daniela},
   title={Graph-theoretical versus topological ends of graphs},
   note={Dedicated to Crispin St. J. A. Nash-Williams},
   journal={J. Combin. Theory Ser. B},
   volume={87},
   date={2003},
   number={1},
   pages={197--206},
   issn={0095-8956},
   review={\MR{1967888}},
   doi={10.1016/S0095-8956(02)00034-5},
}





\bib{inf-cyc-1}{article}{
   author={Diestel, Reinhard},
   author={K\"{u}hn, Daniela},
   title={On infinite cycles I},
   journal={Combinatorica},
   volume={24},
   date={2004},
   number={1},
   pages={69--89},
   issn={1439-6912},
   review={\MR{2057684}},
   doi={10.1007/s00493-004-0005-z},
}



\bib{inf-cyc-2}{article}{
   author={Diestel, Reinhard},
   author={K\"{u}hn, Daniela},
   title={On infinite cycles II},
   journal={Combinatorica},
   volume={24},
   date={2004},
   number={1},
   pages={91--116},
   issn={1439-6912},
   review={\MR{2057685}},
   doi={10.1007/s00493-004-0006-y},
}

\bib{fleisch}{article}{
   author={Fleischner, Herbert},
   title={The square of every two-connected graph is Hamiltonian},
   journal={J. Combinatorial Theory Ser. B},
   volume={16},
   date={1974},
   pages={29--34},
   issn={0095-8956},
   review={\MR{332573}},
   doi={10.1016/0095-8956(74)90091-4}
}




\bib{Freud}{article}{
   author={Freudenthal, Hans},
   title={\"{U}ber die Enden topologischer R\"{a}ume und Gruppen},
   language={German},
   journal={Math. Z.},
   volume={33},
   date={1931},
   number={1},
   pages={692--713},
   issn={0025-5874},
   review={\MR{1545233}},
   doi={10.1007/BF01174375},
}


\bib{agelos}{article}{
   author={Georgakopoulos, Agelos},
   title={Infinite Hamilton cycles in squares of locally finite graphs},
   journal={Adv. Math.},
   volume={220},
   date={2009},
   number={3},
   pages={670--705},
   issn={0001-8708},
   review={\MR{2483226}},
   doi={10.1016/j.aim.2008.09.014}
}




\bib{Jung_countable}{article}{
   author={Jung, H.~A.},
   title={Wurzelb\"{a}ume und unendliche Wege in Graphen},
   language={German},
   journal={Math. Nachr.},
   volume={41},
   date={1969},
   pages={1--22},
   issn={0025-584X},
   review={\MR{266807}},
   doi={10.1002/mana.19690410102},
}


\bib{hamilton_curves}{article}{
   author={K\"{u}ndgen, Andr\'{e}},
   author={Li, Binlong},
   author={Thomassen, Carsten},
   title={Cycles through all finite vertex sets in infinite graphs},
   journal={European J. Combin.},
   volume={65},
   date={2017},
   pages={259--275},
   issn={0195-6698},
   review={\MR{3679848}},
   doi={10.1016/j.ejc.2017.06.006},
}



\bib{bi-power_li}{article}{
   author={Li, Binlong},
   title={Hamiltonicity of bi-power of bipartite graphs, for finite and infinite cases},
   date={2019},
   eprint={1902.06403},
   note={Preprint}
}



\bib{sekanina}{article}{
   author={Sekanina, Milan},
   title={On an ordering of the set of vertices of a connected graph},
   language={English, with Russian summary},
   journal={Spisy P\v{r}\'{\i}rod. Fak. Univ. Brno},
   volume={1960},
   date={1960},
   pages={137--141},
   review={\MR{0140095}}
}





\end{biblist}
\end{bibdiv}
\end{document}